\documentclass[reqno]{amsart}
\usepackage{bbm}
\usepackage{mathrsfs}
\usepackage{color}
\usepackage{amsmath}
\usepackage{amsfonts}
\usepackage{mathtools}
\usepackage{amssymb}
\usepackage{graphicx}
\usepackage{hyperref}
\usepackage{cleveref}

 \newtheorem{Theorem}{Theorem}[section]
 \newtheorem{Corollary}[Theorem]{Corollary}
 \newtheorem{Lemma}[Theorem]{Lemma}
 \newtheorem{Proposition}[Theorem]{Proposition}

\newtheorem{Question}[Theorem]{Question}

 \newtheorem{Remark}[Theorem]{Remark}

 \numberwithin{equation}{section}

\begin{document}

\title[The log-psh of fiberwise $\xi-$Bergman kernels for variant functionals]
 {The log-plurisubharmonicity of fiberwise $\xi-$Bergman kernels for variant functionals}

\author{Shijie Bao}
\address{Shijie Bao: Institute of Mathematics, Academy of Mathematics and Systems Science, Chinese Academy of Sciences, Beijing 100190, China.}
\email{bsjie@amss.ac.cn}

\author{Qi'an Guan}
\address{Qi'an Guan: School of
Mathematical Sciences, Peking University, Beijing 100871, China.}
\email{guanqian@math.pku.edu.cn}

\author{Zheng Yuan}
\address{Zheng Yuan: Institute of Mathematics, Academy of Mathematics and Systems Science, Chinese Academy of Sciences, Beijing 100190, China.}
\email{ yuanzheng@amss.ac.cn}

\thanks{}



\date{}

\dedicatory{}

\commby{}


\begin{abstract}
    In this paper, we establish the log-plurisubharmonicity of fiberwise $\xi$-Bergman kernels for a family of variant functionals, thereby addressing a question posed by Bo Berndtsson to the authors. As an application, we prove that for a plurisubharmonic function $\phi$ and a locally finitely generated ideal sheaf $\mathscr{I}$ on the polydisc $\Delta^{n+m}=\Delta^n\times\Delta^m$, the set of points $w\in\Delta^m$ for which $\mathcal{I}(\phi|_{H_w})_{(o,w)}\subseteq (\mathscr{I}|_{H_w})_{(o,w)}$ is a pluripolar set, where $H_w$ represents the fiber over $w$, and $\mathcal{I}(\phi|_{H_w})$ denotes the multiplier ideal sheaf of $\phi|_{H_w}$ on $H_w$.
\end{abstract}

\maketitle

    \section{Introduction}\label{Introduction}
    The notion of $\xi-$Bergman kernels was introduced in \cite{BG1} to give a new approach to the effectiveness result of the strong openness property of multiplier ideal sheaves. In \cite{BG1, BG2}, Bao-Guan established the log-plurisubharmonicity of fiberwise $\xi-$Bergman kernels for a fixed functional $\xi$, which is a generalization of Berndtsson's log-plurisubharmonicity of fiberwise Bergman kernels (see \cite{Blogsub,Bern09}).

    Using $\xi-$Bergman kernels (combining with \cite{BG24}), in \cite{BG1} and \cite{BG2} respectively, Bao-Guan completed the optimal $L^2$ extension approaches to the sharp effectiveness result of the strong openness property obtained in \cite{Guan19} (a Hilbert bundle approach can be seen in \cite{NW24}), and the sharp effectiveness result of the $L^p$ strong openness property in \cite{GY}.

    By establishing log-plurisubharmonicity of some versions of generalized Bergman kernels, Bao-Guan also achieved some progress on the strong openness properties related to modules at boundary points (see \cite{BGboundary1, BGboundary2, BGboundary3}), which gave new approaches to the sharp effectiveness results related to Jonsson-Musta\c{t}\u{a}'s conjecture in \cite{BGY1, GMY1, GMY2}.

    Considering the asymptotic behaviors of $\xi-$Bergman kernels on sublevel sets of plurisubharmonic functions, Bao-Guan-Yuan introduced the $\xi-$complex singularity exponents in \cite{BGY22}. The relations among $\xi-$complex singularity exponents, jumping numbers, and complex singularity exponents were studied, and Demailly-Koll\'{a}r's restriction formula (\cite{DK01}) and Demailly-Ein-Lazarsfeld's subadditivity property (\cite{DEL00}) of complex singularity exponents were also generalized to $\xi-$complex singularity exponents in \cite{BGY22}.

    The studies listed above concentrated on the fiberwise $\xi-$Bergman kernels with respect to an invariant functional. In a private communication with Bo Berndtsson, he asked the authors the following question:

    \begin{Question}\label{ques-Berndtsson}
    Are the $\xi-$Bergman kernels log-plurisubharmonic with respect to the functional $\xi$?
    \end{Question}
    
    In this paper, we establish the log-plurisubharmonicity of fiberwise $\xi-$Bergman kernels for a family of variant functionals, which answers Question \ref{ques-Berndtsson} affirmatively.

    \subsection{Notations and Conventions}
      We use the notation $\mathbb{N}\coloneqq \mathbb{Z}_{\ge 0}$ in this paper. In \cite{BG1}, the following linear space of sequences of complex numbers was considered:
      \[\ell_1^{(n)}\coloneqq \left\{\xi=(\xi_{\alpha})_{\alpha\in\mathbb{N}^n}  \colon \sum_{\alpha\in\mathbb{N}^n}|\xi_{\alpha}|\rho^{|\alpha|}<+\infty \text{ for any } \rho>0\right\},\]
      where any element in $\ell^{(n)}_1$ can be seen as a linear functional on the ring of germs of holomorphic functions. More precisely, for $z_0\in\mathbb{C}^n$, $\xi=(\xi_{\alpha})\in\ell_1^{(n)}$, and $F(z)=\sum_{\alpha\in\mathbb{N}^n}a_{\alpha}(z-z_0)^{\alpha}\in\mathcal{O}_{\mathbb{C}^n,z_0}$, the value that $\xi$ acts on $F$ is defined as 
      \[(\xi\cdot F)(z_0)\coloneqq \sum_{\alpha\in\mathbb{N}^n}\xi_{\alpha}\frac{F^{(\alpha)}(z_0)}{\alpha!}=\sum_{\alpha\in\mathbb{N}^n}\xi_{\alpha}a_{\alpha}.\]
    
       Now we recall the definition of $\xi-$Bergman kernel. Let $D$ be a domain in $\mathbb{C}^n$, and $\psi$ a plurisubharmonic function on $D$. For any $\xi\in\ell_1^{(n)}$, the (weighted) \emph{$\xi-$Bergman kernel with respect to $\xi$} is denoted by
      \begin{equation*}
          K^{\psi}_{\xi,D}(z)\coloneqq \sup_{f\in A^2(D,e^{-\psi})}\frac{|(\xi\cdot f)(z)|^2}{\int_D|f|^2e^{-\psi}},
      \end{equation*}
      where $A^2(D,e^{-\psi})\coloneqq \{f\in\mathcal{O}(D) \colon \int_D|f|^2e^{-\psi}<+\infty\}$, and we denote $K^{\psi}_{\xi,D}(z)=0$ if $A^2(D,e^{-\psi})=\{0\}$.

      \subsection{Main Result}
    
      In \cite{BG2} (see also \cite{BG1}), the following log-plurisubharmonicity property was established for fiberwise $\xi-$Bergman kernels, which is a generalization of Berndtsson's log-plurisubharmonicity property for fiberwise Bergman kernels (see \cite{Blogsub}). 
      
      Let $\Omega$ be a pseudoconvex domain in $\mathbb{C}^{n+1}$ with coordinate $(z,w)$, where $z\in\mathbb{C}^n$, $w\in\mathbb{C}$. Let $p$ be the natural projection $p(z,w)=w$ on $\Omega$, and $D\coloneqq p(\Omega)$. For any $w\in D$, assume that each $\Omega_w\coloneqq p^{-1}(w)\cap\Omega$ is a bounded domain in $\mathbb{C}^n$. Let $\psi$ be a plurisubharmonic function on $\Omega$, and let $K^{\psi}_{\xi,w}(z)\coloneqq K^{\psi|_{\Omega_{w}}}_{\xi,\Omega_w}(z)$ be the Bergman kernels on the domains $\Omega_w$ with respect to some fixed $\xi\in \ell_1^{(n)}$.
    
     \begin{Theorem}[\cite{BG1, BG2}]\label{xi-log-psh}
        $\log K^{\psi}_{\xi,w}(z)$ is a plurisubharmonic function with respect to $(z,w)$ on $\Omega$.
     \end{Theorem}
     
    In the present paper, we consider the log-plurisubharmonicity of the fiberwise $\xi-$Bergman kernels with respect to variant functionals. We generalize Theorem \ref{xi-log-psh} to the following case.

    Let $\Omega$ be a pseudoconvex domain in $\mathbb{C}^{n+m}$ with coordinate $(z,w)$, where $z\in\mathbb{C}^n$, $w\in\mathbb{C}^m$. Let $p$, $q$ be the natural projections $p(z,w)=w$, $q(z,w)=z$ on $\Omega$, and $D\coloneqq p(\Omega)$. For any $w\in D$, denote $\Omega_w\coloneqq p^{-1}(w)\cap\Omega$. Let $\xi(w)$ be a functional-valued function on $D$, where $\xi(w)=(\xi_{\alpha}(w))_{\alpha\in \mathbb{N}^n}$ and $\xi_{\alpha}(w)\in \ell_1^{(n)}$ for any $w\in D$. We call that $\xi(w)$ is \emph{holomorphic} with respect to $w$ if $\xi_{\alpha}(w)$ is holomorphic with respect to $w\in D$ for each $\alpha\in\mathbb{N}^n$. And we call that $\xi(\cdot)$ satisfies the \emph{locally uniformly bounded property} on $D$, if for any $w\in D$ and any $\rho>0$, there exists a neighborhood $V_{w,\rho}$ of $w$ such that
    \[\sup_{w'\in V_{w,\rho}}\sum_{\alpha\in\mathbb{N}^n}|\xi_{\alpha}(w')|\rho^{|\alpha|}<+\infty\] 
    holds. 

    \begin{Remark}\label{rem-the.only.remark}
    Suppose the functional-valued $\xi(\cdot)$ satisfies the locally uniformly bounded property. Then for any compact subset $K$ of $D$, one has
    \[\sup_{w\in K}\sum_{\alpha\in\mathbb{N}^n}|\xi_{\alpha}(w)|\rho^{|\alpha|}<+\infty\]
    for any $\rho>0$. Therefore, the above neighborhood $V_{w,\rho}$ in the definition of ``locally uniformly bounded property" can be chosen to be independent of $\rho$.

    In addition, if the functional-valued function $\xi(\cdot)$ is holomorphic, and there exists some $N>0$ such that
    $\xi_{\alpha}(w)=0$ for all $\alpha\in\mathbb{N}^n$ with $|\alpha|>N$ and all $w\in D$, then it is clear that $\xi(\cdot)$ satisfies the locally uniformly bounded property on $D$.
    \end{Remark}

    We prove the following theorem, which is the main theorem of the present paper.
    
    \begin{Theorem}[Main Theorem]\label{log-psh-wrt-xi}
    Let $\psi$ be a plurisubharmonic function on $\Omega$, and let $\xi(\cdot)$ be a functional-valued function on $D$ such that $\xi(\cdot)$ is holomorphic and satisfies the locally uniformly bounded property on $D$. Denote $\psi_w\coloneqq \psi|_{\Omega_w}$ for any $w\in D$. Let $K^{\psi_w}_{\xi(w),\Omega_w}(z)$ be the Bergman kernels on the domains $\Omega_w$ with respect to $\xi(w)$. Then $\log K^{\psi_w}_{\xi(w),w}(z)$ is a plurisubharmonic function with respect to $(z,w)$ on $\Omega$.
    \end{Theorem}

    This is a generalization of Theorem \ref{xi-log-psh}. Particularly, if fixing the domains $\Omega_w$ and the plurisubharmonic function $\psi_w$, then we have the following corollary of Theorem \ref{log-psh-wrt-xi}, which affirmatively answers Question \ref{ques-Berndtsson}.

    \begin{Corollary}
    Let $D_1$ be a pseudoconvex domain in $\mathbb{C}^n$, and $D_2$ a domain in $\mathbb{C}^m$. Let $\psi$ be a plurisubharmonic function on $D_1$. Assume that $\xi(\cdot)$ is a functional-valued function on $D_2$ such that $\xi(w)\in \ell^{(n)}_1$ for each $w\in D_2$, and $\xi(\cdot)$ is holomorphic and satisfies the locally uniformly bounded property on $D_2$. The for each $z\in D_1$, the function $\log K_{\xi(w), D_1}^{\psi}(z)$ is a plurisubharmonic function on $D_2$ with respect to $w$. 
    \end{Corollary}

    \subsection{An Application}
    We give an application of Theorem \ref{log-psh-wrt-xi} in the following.

    Let $\Delta^{n+m}_{(z,w)}=\Delta^n_z\times\Delta^m_w$ be the unit polydisc in $\mathbb{C}^{n+m}$ with the origin $o_{(z,w)}=(o_z, o_w)$, where $z$ and $w$ are the coordinates on $\Delta^n_z$ and $\Delta^m_w$ respectively. Denote by $\pi$ the projection $\Delta^n_z\times\Delta^m_w \to \Delta^m_w$, and $H_{w'}\coloneqq \pi^{-1}(w')\cap\Delta^{n+m}_{(z,w)}$ for each $w'\in\Delta^m_w$.

    Let $\mathscr{I}$ be a locally finitely generated ideal sheaf on $\Delta^{n+m}_{(z,w)}$. Then $\mathscr{I}|_{H_{w'}}$, which is $\mathcal{I}$ restricted on the fiber $H_{w'}$, is an ideal sheaf on $H_{w'}$. We can obtain the following corollary of Theorem \ref{log-psh-wrt-xi}.
    
    \begin{Corollary}\label{cor:pluripolar}
        Let $\phi$ be a plurisubharmonic function on $\Delta^{n+m}_{(z,w)}$. Then the set
        \[\Lambda_{\phi,\mathscr{I}}\coloneqq \left\{w'\in\Delta^m_w\colon \mathcal{I}\left(\phi|_{H_{w'}}\right)_{(o_z,w')}\subseteq (\mathscr{I}|_{H_{w'}})_{(o_z,w')}\right\}\]
        is either the whole $\Delta_w^m$ or a pluripolar subset of $\Delta_w^m$.

        Moreover, if $(\mathscr{I}|_{H_{w'}})_{(o_z,w')}=I$ for every $w'\in\Delta^m_w$, where $I$ is an ideal of $\mathcal{O}_{\mathbb{C}^n,o}$, then the set $\Lambda(\phi,\mathscr{I})$ is either the whole $\Delta_w^m$ or a complete pluripolar subset of $\Delta_w^m$.
    \end{Corollary}

    Here
    \[\mathcal{I}\left(\phi|_{H_{w'}}\right)_{(o_z,w')}\coloneqq \left\{f\in\mathcal{O}_{H_{w'},(o_z,w')}\colon |f|^2e^{-\phi|_{H_{w'}}} \text{ is integrable near } (o_z,w')\right\}\]
    denotes the stalk of the multiplier ideal sheaf of $\phi|_{H_{w'}}$ at the point $(o_z,w')$ on $H_{w'}$.

    Particularly, if for every $w'\in\Delta^m_w$, $(\mathscr{I}|_{H_{w'}})_{(o_z,w')}$ equals the maximal ideal of the local ring $\mathcal{O}_{H_{w'}, (o_z,w')}$, then due to the openness property of multiplier ideal sheaves (see \cite{Bern13, GZ15}) we have
    \[\Lambda_{\phi,\mathscr{I}}=\left\{w'\in\Delta^m_w\colon c_{(o_z,w')}\left(\phi|_{H_{w'}}\right)\ge 1\right\},\]
    where $c_{(o_z,w')}\left(\phi|_{H_{w'}}\right)$ is the complex singularity exponent of $\phi|_{H_{w'}}$ at the point $(o_z,w')$. It has been shown in \cite[Lemma 4.6]{GZ2020} that the result in Corollary \ref{cor:pluripolar} holds in this special case due to Berndtsson's log-plurisubharmonicity of fiberwise classical Bergman kernels.

    \section{Preparations}
    We recall some basic lemmas in \cite{BG1, BG2}. Lemma \ref{lem-holo-wrt-z} shows that a holomorphic function  remains holomorphic after applying functionals from $\ell^{(n)}_1$.

    \begin{Lemma}[\cite{BG1}]\label{lem-holo-wrt-z}
        Let $D$ be a domain in $\mathbb{C}^n$, and let $F$ be a holomorphic function on $D$. Then $(\xi\cdot F)(z)$ is also a holomorphic function on $D$, where $\xi\in \ell_1^{(n)}$.
    \end{Lemma}

    Lemma \ref{lem-bounded} indicates that any functional in $\ell_1^{(n)}$ forms a bounded operator on the Hilbert space $A^2(D,e^{-\psi})$.

    \begin{Lemma}[\cite{BG2}]\label{lem-bounded}
        Let $D$ be a domain in $\mathbb{C}^n$, $\xi\in\ell_1^{(n)}$, and $\psi$ a plurisubharmonic function on $D$. Then for any compact subset $K$ of $D$, there exists a positive constant $C_K$, such that
        \[|(\xi\cdot f)(z)|^2\leq \int_{D}|f|^2e^{-\psi},\]
        for any $z\in D$ and $f\in A^2(D,e^{-\psi}).$
    \end{Lemma}

    The following Lemma gives the existence of the holomorphic function achieving the infimum in the definition of $\xi-$Bergman kernel.
    
    \begin{Lemma}[\cite{BG2}]\label{lem-F0}
        Let $D$ be a domain in $\mathbb{C}^n$, $z\in D$, and $\psi$ a plurisubharmonic function on $D$ with $A^2(D,e^{-\psi})\neq\{0\}$. Then for any $\xi\in \ell_1^{(n)}$, there exists a holomorphic function $F_0$ on $D$ such that
        \[K^{\psi}_{\xi,D}(z)=\frac{|(\xi\cdot F_0)(z)|^2}{\int_D |F_0|^2e^{-\psi}}.\]
    \end{Lemma}

    The following is a basic log-plurisubharmonicity of $\xi-$Bergman kernels.

    \begin{Lemma}[\cite{BG2}]\label{lem-log-psh-wrt-z}
        Let $D$ be a domain in $\mathbb{C}^n$, $\xi\in \ell_1^{(n)}$, and $\psi$ a plurisubharmonic function on $D$. Then $\log K^{\psi}_{\xi,D}(z)$ is plurisubharmonic on $D$.
    \end{Lemma}

    We prove some additional lemmas. In Lemma \ref{lem-conti} and Lemma \ref{lem-hol}:
    
    we let $\Omega$ be a domain in $\mathbb{C}^{n+m}$. Denote by $p,q$ the natural projections from $\mathbb{C}^{n+m}$ to $\mathbb{C}^m$ and $\mathbb{C}^n$ respectively. Set $D\coloneqq p(\Omega)$, and $\Omega_w\coloneqq p^{-1}(w)\cap\Omega$ for $w\in D$. Let $\xi(\cdot)$ be a functional-valued function on $D$, such that $\xi(w)\in\ell_1^{(n)}$ for every $w\in D$.
    
    \begin{Lemma}\label{lem-conti}
        Suppose the functional-valued function $\xi(\cdot)$ satisfies the locally uniformly bounded property near some $w_0\in D$, and $\xi_{\alpha}(\cdot)$ is continuous at $w_0$ for every $\alpha\in\mathbb{N}^n$. Let $(z_0,w_0)\in \Omega$, and $V\subset\subset q(\Omega_{w_0})$ an open subset such that $z_0\in V$. Let $\{(z_j,w_j)\}\subset \Omega$ be a sequence of points such that $z_j\in V$ for each $j\in\mathbb{Z}_+$, and $(z_j,w_j)\to (z_0,w_0)$ as $j\to +\infty$. If $\{f_j\}_{j\in\mathbb{Z}_+}$ is a sequence of holomorphic functions on $q(\Omega_{w_j})$ respectively satisfying that $f_j$ compactly converges to $f\in \mathcal{O}(V)$ on $V$, then
        \[\lim_{j\to +\infty} (\xi(w_j)\cdot f_j)(z_j)=(\xi(w_0)\cdot f)(z_0). \]
    \end{Lemma}

    \begin{proof}
        Let $U\subset\subset q(\Omega_{w_0})$, and we may assume $U\subset\subset q(\Omega_{w_j})$ for every $j\ge 0$. Since $f_j$ compactly converges to $f$ on $V$, by Cauchy's inequality, there exist some $M>0$ and $R>0$, such that
        \[\frac{|f_j^{(\alpha)}(z)|}{\alpha!}\leq \frac{M}{R^{|\alpha|}},\ \forall z\in U, j\ge 0, \alpha\in\mathbb{N}^n.\]
        As $\xi(\cdot)$ is locally uniformly bounded near $w_0$, we have
        \[\sup_{j\ge 0}\sum_{\alpha\in\mathbb{N}^n}|\xi_{\alpha}(w_j)|\rho^{|\alpha|}<+\infty\]
        for some fixed $\rho>1/R$. Then for any given $\epsilon>0$, there exists $k\in\mathbb{N}$ such that 
        \begin{equation}\label{sum|alpha|>k}
            \begin{split}
                \sup_{j\ge 0}\sum_{|\alpha|>k}|\xi_{\alpha}(w_j)|\frac{|f_j^{(\alpha)}(z_j)|}{\alpha!}&\leq \sup_{j\ge 0}\sum_{|\alpha|>k}|\xi_{\alpha}(w_j)|\frac{M}{R^{|\alpha|}}\\
                &\leq \frac{1}{(\rho R)^k}\sup_{j\ge 0}\sum_{|\alpha|>k}|\xi_{\alpha}(w_j)|\rho^{|\alpha|}<\epsilon.
            \end{split}
        \end{equation}
    It follows $f_j$ compactly converges to $f$ that 
    \[\lim_{j\to +\infty}f^{(\alpha)}_j(z_j)= f^{(\alpha)}(z_0), \ \forall \alpha\in\mathbb{N}^n.\]
    Combining with the continuities of $\xi_{\alpha}(\cdot)$'s at $w_0$, we get
    \[\lim_{j\to +\infty}\sum_{|\alpha|\leq k}\xi_{\alpha}(w_j)\frac{f_j^{(\alpha)}(z_j)}{\alpha!}=\sum_{|\alpha|\leq k}\xi_{\alpha}(w_0)\frac{f^{(\alpha)}(z_0)}{\alpha!}.\]
    Now according to (\ref{sum|alpha|>k}) we have
    \begin{equation*}
        \begin{split}
            \lim_{j\to +\infty} (\xi(w_j)\cdot f_j)(z_j)&=\lim_{j\to +\infty}\sum_{\alpha\in\mathbb{N}^n}\xi_{\alpha}(w_j)\frac{f_j^{(\alpha)}(z_j)}{\alpha!}\\
            &=\sum_{\alpha\in\mathbb{N}^n}\xi_{\alpha}(w_0)\frac{f^{(\alpha)}(z_0)}{\alpha!}=(\xi(w_0)\cdot f)(z_0).
        \end{split}
    \end{equation*}
    \end{proof}

        The following lemma can be seen as a fiberwise version of Lemma \ref{lem-holo-wrt-z}.

    \begin{Lemma}\label{lem-hol}
        Suppose the functional-valued function $\xi(\cdot)$ is holomorphic and satisfies the locally uniformly bounded properties near $w_0\in D$. Let $z_0\in \Omega_{w_0}$. Then for any holomorphic function $F(z,w)$ near $(z_0,w_0)$, $(\xi(w)\cdot F_w)(z)$ is holomorphic with respect to $(z,w)$ near $(z_0,w_0)$, where $F_w(\cdot)\coloneqq F(\cdot, w)=F|_{\Omega_w}$.
    \end{Lemma}
    
\begin{proof}
    It is known that $(\xi(w)\cdot F_w)(z)$ is holomorphic with respect to $z$ for any fixed $w$ by Lemma \ref{lem-holo-wrt-z}. Now we prove that $(\xi(w)\cdot F_w)(z^*)$ is holomorphic with respect to $w$ for any fixed $z^*$ near $z_0$. Note that
    \[(\xi(w)\cdot F_w)(z^*)=\sum_{\alpha\in\mathbb{N}^n}\xi_{\alpha}(w)c_{\alpha}(w),\]
    where
    \[c_{\alpha}(w)\coloneqq \left.\frac{1}{\alpha!}\cdot\frac{\partial ^{|\alpha|} F(z,w)}{\partial z^{\alpha}}\right|_{z=z^*}\]
    is holomorphic with respect to $w$ for each $\alpha\in\mathbb{N}^n$. We also have that $\xi_{\alpha}(w)$ is holomorphic with respect to $w$ since $\xi(\cdot)$ is holomorphic. 
    
    Next we prove that the summation $\sum_{\alpha\in\mathbb{N}^n}\xi_{\alpha}(w)c_{\alpha}(w)$ is uniformly convergent near $w_0$. Let $V\subset\subset q(\Omega_{w_0})$ be an open neighborhood of $z^*$, and we can find some small open neighborhood $W$ of $w_0$ such that $V\subset\subset q(\Omega_w)$ for any $w\in W$. Then according to Cauchy's inequality we can find some $R>0$ and $M>0$ such that
    \[|c_{\alpha}(w)|\leq\frac{M}{R^{|\alpha|}}, \ \forall w\in W, \alpha\in\mathbb{N}^n.\]
    Since $\xi(\cdot)$ satisfies the locally uniformly bounded property, we can shrink $W$ if necessarily, such that
    \[\sup_{w\in W}\sum_{\alpha\in\mathbb{N}^n}|\xi_{\alpha}(w)|\rho^{|\alpha|}<+\infty\]
    for some fixed $\rho>1/R$. It follows that there exists some $M'>0$ such that
    \[|\xi_{\alpha}(w)|\leq \frac{M'}{\rho^{|\alpha|}}, \ \forall w\in W, \alpha\in\mathbb{N}^n.\]
    Thus, 
    \[|\xi_{\alpha}(w)c_{\alpha}(w)|\leq \frac{MM'}{(\rho R)^{|\alpha|}}, \ \forall w \in W, \alpha\in\mathbb{N}^n.\]
    Since $\rho R>1$, we have $\sum_{\alpha\in\mathbb{N}^n}\frac{MM'}{(\rho R)^{|\alpha|}}<+\infty$. Then we obtain that the summation $\sum_{\alpha\in\mathbb{N}^n}\xi_{\alpha}(w)c_{\alpha}(w)$ is uniformly convergent on $W$, and it follows that $(\xi(w)\cdot F_w)(z^*)=\sum_{\alpha\in\mathbb{N}^n}\xi_{\alpha}(w)c_{\alpha}(w)$ is holomorphic near $w_0$ with respect to $w$ since all $\xi_{\alpha}(w)$ and $c_{\alpha}(w)$ are holomorphic with respect to $w$.

    Now according to Hartogs' Theorem, we get that $(\xi(w)\cdot F_w)(z)$ is holomorphic with respect to $(z,w)$ near $(z_0,w_0)$.
\end{proof}

The following proposition shows that the assumptions for the functional-valued function in Lemma \ref{lem-hol} are necessary.

    \begin{Proposition}
        Let $D$ be a domain in $\mathbb{C}^m$, and $\xi(\cdot)$ a functional-valued function on $D$ such that $\xi(w)\in\ell^{(n)}_1$ for any $w\in D$. If for any bounded domain $\Omega'\subset\mathbb{C}^n$, any $z_0\in \Omega'$, and any $F\in A^2(\Omega')$, we have that $(\xi(w)\cdot F)(z_0)$ is holomorphic with respect to $w\in D$, then the functional-valued function $\xi(\cdot)$ is holomorphic and satisfies the locally uniformly bounded property on $D$. 
    \end{Proposition}

    \begin{proof}
    Let $R>0$, and $\Omega'=\Delta^n_R\subset\mathbb{C}^n$ the polydisc centered at the origin $o$ with radius $R$. Let $z=(z_1,\cdots,z_n)$ be the coordinate on $\Omega'$.
    
    First, note that $z^{\alpha}\coloneqq z_1^{\alpha_1}\cdots z_n^{\alpha_n}\in A^2(\Omega')$ for any $\alpha=(\alpha_1,\ldots,\alpha_n)\in\mathbb{N}^n$. It follows the assumption for $\xi(\cdot)$ that $\xi_{\alpha}(w)=(\xi(w)\cdot z^{\alpha})(o)$ is holomorphic with respect to $w\in D$ for each $\alpha\in\mathbb{N}^n$. Then we get that $\xi(w)$ is holomorphic with respect to $w\in D$.

    Next, for any $w_0\in D$ we can find some neighborhood $V_{w_0}$ of $w_0$ in $D$, such that the family $\{\xi(w)\}_{w\in V_{w_0}}$ of linear continuous (by Lemma \ref{lem-bounded}) functionals over $A^2(\Omega')$ satisfies
    \[\sup_{w\in V_{w_0}}|(\xi(w)\cdot F)(o)|<+\infty\]
    for any fixed $F\in A^2(\Omega')$, since $(\xi(w)\cdot F)(o)$ is holomorphic on $D$. Now the uniformly bounded principle (see e.g. \cite{FA-Rudin}) implies that for any bounded subset $\mathcal{U}$ of $A^2(\Omega')$, there exists a constant $M_{\mathcal{U}}>0$, such that
    \begin{equation}\label{UBP}
        \sup_{h\in \mathcal{U}}|\xi(w)\cdot h(o)|\leq M_{\mathcal{U}}, \ \forall w\in V_{w_0}.
    \end{equation}
    We choose the bounded subset $\mathcal{U}$ of $A^2(\Omega')=A^2(\Delta^n_R)$ as follows:
    \[\mathcal{U}\coloneqq \left\{h=\sum_{\alpha\in\mathbb{N}^n}a_{\alpha}z^{\alpha}\in\mathcal{O}(\Delta^n_R) \colon |a_{\alpha}|=\frac{1}{(2R)^{|\alpha|}},\ \forall \alpha\in\mathbb{N}^n\right\}.\]
    For each $h\in\mathcal{U}$, we have
    \begin{align*}
    \begin{split}
    \int_{\Delta^n_R}|h|^2&=\sum_{\alpha\in\mathbb{N}^n}\frac{\pi^n|a_{\alpha}|^2}{(\alpha_1+1)\cdots(\alpha_n+1)}R^{2(|\alpha|+n)}\\
    &= \sum_{\alpha\in\mathbb{N}^n}\frac{\pi^nR^{2n}}{(\alpha_1+1)\cdots(\alpha_n+1)}\left(\frac{1}{2}\right)^{|\alpha|}<+\infty,
        \end{split}
    \end{align*}
    yielding that $\mathcal{U}$ is exactly a bounded subset of $A^2(\Omega')$. Now (\ref{UBP}) gives that
    \[\sup_{|a_{\alpha}|=\frac{1}{(2R)^{|\alpha|}} ( \forall\alpha\in\mathbb{N}^n)}\left|\sum_{\alpha\in\mathbb{N}^n}\xi_{\alpha}(w)a_{\alpha}\right|\leq M,\ \forall w\in V_{w_0},\]
    for some constant $M>0$ independent of $w$. It follows that
    \[\sup_{w\in V_{w_0}}\sum_{\alpha\in\mathbb{N}^n}|\xi_{\alpha}(w)|\frac{1}{(2R)^{|\alpha|}}<+\infty.\]
    Since $R$ and $w_0\in D$ are arbitrary, we conclude that $\xi(\cdot)$ satisfies the locally uniformly bounded property on $D$.
    \end{proof}

    \section{Proof of Theorem \ref{log-psh-wrt-xi}}

    For the proof of Theorem \ref{log-psh-wrt-xi}, we also need to recall the following version of the optimal $L^2$ extension theorem.

    \begin{Lemma}[Optimal $L^2$ extension theorem, see \cite{Bl13,GZ-L2ext}]\label{optimal-L2-ext}
        Let $\Omega$ be a pseudoconvex domain in $\mathbb{C}^{n+1}$ with coordinate $(z,w)$, where $z\in\mathbb{C}^n$, $w\in\mathbb{C}$, and $p$ the natural projection $p(z,w)=w$ on $\Omega$. Let $D\coloneqq p(\Omega)$ and assume that $D=\Delta_{w_0,r}$ is the disc in the complex plane centered at $w_0$ with radius $r$. For any $w\in D$, denote $\Omega_w\coloneqq p^{-1}(w)\subseteq\Omega$. Let $\psi$ be a plurisubharmonic function on $\Omega$.
        Then for every $f\in A^2(\Omega_{w_0},e^{-\psi|_{\Omega_{w_0}}})$, there exists a holomorphic function $F$ on $\Omega$, such that $F|_{\Omega_{w_0}}=f$, and
        \[\frac{1}{\pi r^2}\int_{\Omega}|F|^2e^{-\psi}\leq \int_{\Omega_{w_0}}|f|^2e^{-\psi|_{\Omega_{w_0}}}.\]
    \end{Lemma}

    Now we prove Theorem \ref{log-psh-wrt-xi} by using Lemma \ref{optimal-L2-ext} (called Guan-Zhou Method, see \cite{GZ-L2ext,Oh-book1,Oh-book2}).

    \begin{proof}[Proof of Theorem \ref{log-psh-wrt-xi}.]
    We firstly prove $\log K^{\psi_w}_{\xi(w),w}(z)$ is upper-semicontinuous. Let $(z_0,w_0)\in\Omega$. Let $\{(z_j,w_j)\}\subset\Omega$ be a sequence of points such that $(z_j,w_j)\rightarrow (z_0,w_0)$ as $j\to +\infty$, and
    \[K^{\psi_{w_j}}_{\xi(w_j),w_j}(z_j)\to \limsup_{\Omega\ni(z,w)\to (z_0,w_0)}K^{\psi_w}_{\xi(w),w}(z), \ \text{when }j\to+\infty.\]
    Without loss of generality, we assume $\log K^{\psi_{w_j}}_{\xi(w_j),w_j}(z_j)>-\infty$ for any $j$. According to Lemma \ref{lem-F0}, for any $j\in\mathbb{Z}_+$, there exists some $f_j\in A^2(\Omega_{w_j}, e^{-\psi_{w_j}})$ such that
    \[|(\xi(w_j)\cdot f_j)(z_j)|^2=K^{\psi_{w_j}}_{\xi(w_j),w_j}(z_j), \ \text{and }\int_{\Omega_{w_j}}|f_j|^2e^{-\psi_{w_j}}=1.\]
    Using Montel's Theorem, we can extract a subsequence of $q_*(f_{j})$ (denoted by $q_*(f_j)$ itself) which uniformly converges to some $f_0\in\mathcal{O}(q(\Omega_{w_0}))$ on any compact subset of $q(\Omega_{w_0})$. Then Lemma \ref{lem-conti} implies
    \[(\xi(w_0)\cdot f_0)(z_0)=\lim_{j\to+\infty}(\xi(w_j)\cdot f_j)(z_j)=\limsup_{\Omega\ni(z,w)\to (z_0,w_0)}K^{\psi_w}_{\xi(w),w}(z).\]
    In addition, it follows from Fatou's Lemma and the upper-semicontinuity of $\psi$ that
    \[\int_{q(\Omega_{w_0})}|f_0|^2e^{-\psi_{w_0}}\leq\liminf_{j\to+\infty}\int_{\Omega_{w_j}}|f_j|^2e^{-\psi_{w_j}}=1.\]
    Thus, we get
    \[K^{\psi_{w_0}}_{\xi(w_0),w_0}(z_0)\geq\frac{|(\xi(w_0)\cdot f_0)(z_0)|^2}{\int_{q(\Omega_{w_0})}|f_0|^2e^{-\psi_{w_0}}}\geq \limsup_{\Omega\ni(z,w)\to (z_0,w_0)}K^{\psi_w}_{\xi(w),w}(z).\]
    It means that $K^{\psi_w}_{\xi(w),w}(z)$ is upper-semicontinuous, which gives that $\log K^{\psi_w}_{\xi(w),w}(z)$ is also upper-semicontinuous on $\Omega$.

    Then we need to prove that $\log K^{\psi_w}_{\xi(w),w}(z)$ satisfies the mean value inequality on any complex line $L\subset \Omega$. If $L$ lies on $\Omega_{w}$ for some $w\in D$, Lemma \ref{lem-log-psh-wrt-z} shows it is true. Then without loss of generality, we may assume $L$ lies on some $q^{-1}(z_0)$, and $L=\Delta(w_0,r)$ is a disc for some $w_0\in D$ and $r>0$.

    According to Lemma \ref{lem-F0}, there exists some $f\in A^2(\Omega_{w_0},e^{-\psi_{w_0}})$ such that
    \begin{equation}\label{eq-f}
        K^{\psi_{w_0}}_{\xi(w_0),w_0}(z_0)=\frac{|(\xi(w_0)\cdot f)(z_0)|^2}{\int_{\Omega_{w_0}}|f|^2e^{-\psi_{w_0}}}.
    \end{equation}
    Next, using Lemma \ref{optimal-L2-ext}, we can get a holomorphic function $F$ on $p^{-1}(\Delta(w_0,r))$ such that $F|_{\Omega_{w_0}}=f$ and
    \[\frac{1}{\pi r^2}\int_{p^{-1}(\Delta(w_0,r))}|F|^2e^{-\psi}\leq \int_{\Omega_{w_0}}|f|^2e^{-\psi_{w_0}}.\]
    Now it follows from Jensen's inequality and Fubini's Theorem that
    \begin{equation}\label{eq-Jensen}
        \begin{split}
            \log\left(\int_{\Omega_{w_0}}|f|^2e^{-\psi_{w_0}}\right)&\geq \log\left(\frac{1}{\pi r^2}\int_{p^{-1}(\Delta(w_0,r))}|F|^2e^{-\psi}\right)\\
            &=\log \left(\frac{1}{\pi r^2}\int_{w\in\Delta(w_0,r)}\int_{\Omega_w}|F_w|^2e^{-\psi_w}\right)\\
            &\geq\frac{1}{\pi r^2}\int_{w\in\Delta(w_0,r)}\log\left(\int_{\Omega_w}|F_w|^2e^{-\psi_w}\right)\\
            &\geq\frac{1}{\pi r^2}\int_{w\in\Delta(w_0,r)}\Big(\log|(\xi(w)\cdot F_w)(z_0)|^2-\log K^{\psi_w}_{\xi(w),w}(z_0)\Big),
        \end{split}
    \end{equation}
    where $F_w(\cdot)\coloneqq F(\cdot,w)=F|_{\Omega_w}$ for every $w\in\Delta(w_0,r)$. Lemma \ref{lem-hol} tells us that $(\xi(w)\cdot F_w)(z_0)$ is holomorphic with respect to $w$, which implies that $\log|(\xi(w)\cdot F_w)(z_0)|^2$ is subharmonic with respect to $w$, yielding
    \begin{equation}\label{eq-subharmonic}
        \log|(\xi(w_0)\cdot F_{w_0})(z_0)|^2\leq \frac{1}{\pi r^2}\int_{w\in\Delta(w_0,r)}\log|(\xi(w)\cdot F_w)(z_0)|^2.
    \end{equation}
    Now combining (\ref{eq-f}), (\ref{eq-Jensen}) and (\ref{eq-subharmonic}), we obtain
    \[\log K^{\psi_w}_{\xi(w),w}(z_0)\leq \frac{1}{\pi r^2}\int_{w\in\Delta(w_0,r)}\log K^{\psi_w}_{\xi(w),w}(z_0),\]
    which gives the mean value inequality of $\log K^{\psi_w}_{\xi(w),w}(z)$ along $L=\Delta(w_0,r)$.

    Now we can conclude that $\log K^{\psi_w}_{\xi(w),w}(z)$ is plurisubharmonic on $\Omega$ with respect to $(z,w)$.
    \end{proof}

    \section{Proof of Corollary \ref{cor:pluripolar}}
    In this section we prove Corollary \ref{cor:pluripolar}. We start from some lemmas which will be used.

    The following well-known Krull's Lemma will allow us to reduce some questions from $\mathcal{O}_{\mathbb{C}^n,o}$ to some finite dimensional linear spaces.

    \begin{Lemma}[Krull's lemma, see \cite{AM}]\label{lem-Krull.lem}
        Let $R$ be a Noetherian local ring with the unique maximal ideal $\mathfrak{m}$. Then for any ideal $I$ of $R$,
        \[\bigcap_{k\in\mathbb{N}_+}(I+\mathfrak{m}^k)=I.\]
    \end{Lemma}

    The following is a lemma from linear algebra.

    \begin{Lemma}\label{lem-linear.algebra}
        Let $p\le q$ be positive integers. Suppose that
        \[\mathbf{A}(w)=\Big(a_{ij}(w)\Big)_{\substack{1\le i\le p \\ 1\le j\le q}}\]
        is a matrix for $w\in\Delta^m$, where each $a_{ij}(w)$ is a holomorphic function in $w$. Set
        \[r(w)\coloneqq \mathrm{rank}\big(\mathbf{A}(w)\big)\in\mathbb{N}\cap [0, p], \ r\coloneqq \max_{w\in\Delta^m}r(w).\]
        Then there exists an open subset $U$ of $\Delta^m$ and a matrix
        \[\mathbf{B}(w)=\Big(b_{kl}(w)\Big)_{\substack{1\le k\le p-r\\ 1\le l\le p}}\]
        on $\Delta^m$ such that

        (1) $\Delta^m\setminus U$ is an analytic subset of $\Delta^m$;

        (2) each $b_{kl}(w)$ is a holomorphic function in $w\in \Delta^m$;
        
        (3) for every $w\in U$, the sequence
        \[\mathbb{C}^q\xrightarrow{\mathbf{A}(w)\cdot}\mathbb{C}^p\xrightarrow{\mathbf{B}(w)\cdot}\mathbb{C}^{p-r}\longrightarrow 0\]
        is exact, where $\mathbb{C}^q, \mathbb{C}^p, \mathbb{C}^{p-r}$ are seen as the spaces of column vectors.
    \end{Lemma}

    \begin{proof}
        Since $r=\max_{w\in\Delta^m}r(w)$, there must be some $w_0\in\Delta^m$ and an $r\times r$ submatrix of $\mathbf{A}(w_0)$ with a nonzero determinant. Without loss of generality, we may assume that the submatrix
        \begin{equation*}\mathbf{C}(w)\coloneqq 
            \begin{pmatrix}
            a_{11}(w) & \cdots & a_{1r}(w) \\
            \vdots & \ddots & \vdots \\
            a_{r1}(w) & \cdots & a_{rr}(w)
        \end{pmatrix}
    \end{equation*}
    of $\mathbf{A}(w)$ satisfies $\det(\mathbf{C}(w_0))\neq 0$. Set
    \[U\coloneqq \left\{w\in\Delta^m\colon \det\big(\mathbf{C}(w)\big)\neq 0\right\}.\]
    Then $U$ is an open subset of $\Delta^m$ and $\Delta^m\setminus U$ is an analytic subset of $\Delta^m$. In the following, we construct the matrix $\mathbf{B}(w)$.

    Let $w\in \Delta^m$. For each $j\in \{1, 2, \ldots, p-r\}$, by the definition of the rank of a matrix, the $(r+1)\times (r+1)$ submatrix
    \begin{equation*}\mathbf{D}_j(w)\coloneqq 
        \begin{pmatrix}
        a_{11}(w) & \cdots & a_{1r}(w) & a_{1,r+j}(w)\\
        \vdots & \ddots & \vdots & \vdots\\
        a_{r1}(w) & \cdots & a_{rr}(w) & a_{r,r+j}(w) \\
        a_{j1}(w) & \cdots & a_{jr}(w) & a_{r+j,r+j}(w) 
    \end{pmatrix}
\end{equation*}
    of $\mathbf{A}(w)$ must satisfy $\det(\mathbf{D}_j(w))=0$. Set $d^k_{j}(w)$ to be the cofactor of the element $a_{k,r+j}(w)$ in $\mathbf{D}_j(w)$ for each $k\in\{1,\ldots,r\}\cup\{r+j\}$. Then $d^k_{j}(w)$'s are holomorphic in $w$ on $\Delta^m$, and $d_{j}^{r+j}(w)=\det(\mathbf{C}(w))\neq 0$ for every $w\in U$.
    
    It follows from the basic properties of determinants that
    \[\sum_{k=1}^{r}a_{kl}(w)d_{j}^k(w)+a_{jl}(w)d_{j}^{r+j}(w)=0, \ \forall l\in\{1,\ldots,r\}.\]
    In addition, since $\mathrm{rank}(\mathbf{A}(w))\le r$ and $\det(\mathbf{C}(w))\neq 0$ for $w\in U$, all the $(r+1)$-th to $q-$th column vectors of $\mathbf{A}(w)$ are linear combinations of the first $r$ column vectors of $\mathbf{A}(w)$ when $w\in U$, which implies that
    \[\sum_{k=1}^{r}a_{kl}(w)d_{j}^k(w)+a_{jl}(w)d_{j}^{r+j}(w)=0, \ \forall l\in\{1,\ldots,q\}, \ w\in U.\]

    For each $j\in\{1,\ldots,p-r\}$ and $w\in\Delta^m$, let
    \begin{equation*}
        \mathbf{X}_{j}(w)\coloneqq \big(d_{j}^{1}(w),\cdots,d_{j}^r(w), 0,\cdots, 0, d_{j}^{r+j}(w),0,\cdots,0\big)\in\mathbb{C}^{1\times p},
    \end{equation*}
     where $d_{j}^{r+j}(w)$ is at the $(r+j)-$th position. Denote 
    \[L_w\coloneqq \big\{\mathbf{X}\in\mathbb{C}^{1\times p}\colon \mathbf{X}\cdot \mathbf{A}(w)=\mathbf{0}\big\},\]
    where $w\in\Delta^m$. Then for $w\in U$, we have $\dim(L_w)=p-\mathrm{rank}(\mathbf{A}(w))=p-r$. Now according to the above discussions, we have that $\mathbf{X}_j(w)\in L_w$ and $\mathbf{X}_j(w)$ is holomorphic in $w$ on $\Delta^m$ for each $j$. Since $d_{j}^{r+j}(w)\neq 0$ on $U$, it can be seen that $\mathbf{X}_1(w),\ldots, \mathbf{X}_{p-r}(w)$ are linear independent for every fixed $w\in U$. Thus, $\{\mathbf{X}_1(w),\ldots, \mathbf{X}_{p-r}(w)\}$ forms a holomorphic basis of $L_w$ when $w\in U$.

     Finally, let
     \begin{equation*}\mathbf{B}(w)=
        \begin{pmatrix}
            \mathbf{X}_1(w)\\
            \vdots\\
            \mathbf{X}_{p-r}(w)
        \end{pmatrix}, \ w\in \Delta^m.
     \end{equation*}
    Then every element in the matrix $\mathbf{B}(w)$ is holomorphic in $w\in\Delta^m$, and $\mathbf{B}(w)\cdot\mathbf{A}(w)=\mathbf{O}_{(p-r)\times q}$ when $w\in U$. It follows that $\mathrm{Ker}(\mathbf{B}(w))\supseteq\mathrm{Im}(\mathbf{A}(w))$ for each $w\in U$. Also, we have $\mathrm{rank}(\mathbf{B}(w))=p-r=\dim(\mathbb{C}^p)-\dim(\mathrm{Im}\,\mathbf{A}(w))$, yielding that $\mathrm{Ker}(\mathbf{B}(w))=\mathrm{Im}(\mathbf{A}(w))$ for each $w\in U$, which shows that the sequence in the statement (3) is exact for every $w\in U$.

    The proof is done.
    \end{proof}

    Recall that $\pi$ is the projection $\Delta^{n+m}\to\Delta_w^m$ and $H_w\coloneqq \pi^{-1}(w)\cap\Delta^{n+m}$ is the fiber over each $w\in\Delta^m$. Let $\mathscr{I}$ be an ideal sheaf on $\Delta^{n+m}$ \textbf{globally} generated by some $F_1,\cdots,F_t\in\mathcal{O}(\Delta^{n+m})$. Denote $I_w\coloneqq (\mathscr{I}|_{H_{w}})_{(o,w)}$ for each $w\in\Delta^m$. In addition, for $\xi\in\ell_1^{(n)}$, denote
    \[\deg(\xi)\coloneqq \sup\big\{k\colon \exists \alpha\in\mathbb{N}^n \text{ with } |\alpha|=k \text{ and } \xi_{\alpha}\neq 0\big\}\in\mathbb{N}\cup\{+\infty\}.\]
    We prove the following lemma by Lemma \ref{lem-linear.algebra}.

    \begin{Lemma}\label{lem-construct.xi-i}
    Assume that there exists a positive integer $N$ satisfying $\mathfrak{m}^N\subseteq I_{w}$ for every $w\in\Delta^m$, where $\mathfrak{m}$ is the maximal ideal of $\mathcal{O}_{\mathbb{C}^n,o}$.

    Then there exists an open subset $U$ of $\Delta^m$ and a collection of functional-valued functions $\{\xi_1(\cdot),\ldots,\xi_s(\cdot)\}$ on $\Delta^m$ such that

    (i) $\Delta^m\setminus U$ is an analytic subset of $\Delta^m$;
    
    (ii) $\xi_i(w)\in\ell_1^{(n)}$ and $\deg(\xi_i(w))\le N-1$, for each $i\in\{1,\ldots,s\}$ and $w\in \Delta^m$;

    (iii) each $\xi_i(\cdot)$ is holomorphic on $\Delta^m$, $i\in\{1,\ldots,s\}$;

    (iv) for every $w\in U$,
    \begin{equation}\label{eq-I.zeros.of.xi}
        I_{w}=\big\{f\in\mathcal{O}_{\mathbb{C}^n,o}\colon (\xi_i(w)\cdot f)(o)=0, \ i=1,\ldots,s\big\}.
    \end{equation}
    \end{Lemma}

    \begin{proof}
    Denote by $\Pi$ the quotient map $\mathcal{O}_{\mathbb{C}^n,o}\to \mathcal{O}_{\mathbb{C}^n,o}/\mathfrak{m}^N$. Note that $\mathcal{V}\coloneqq \mathcal{O}_{\mathbb{C}^n,o}/\mathfrak{m}^N$ can be seen as a finite dimensional linear space with the basis
    \[\mathcal{B}\coloneqq \{z^{\alpha}\colon \alpha\in\mathbb{N}^n, \ |\alpha|\le N-1\},\]
    where $z$ is the coordinate on $\mathbb{C}^n$. For each $w\in\Delta^m$, let $\mathcal{L}_w\coloneqq \Pi(I_w)$, which is a linear subspace of $\mathcal{V}$. By the assumption of $\mathscr{I}$, each ideal $I_w$ is generated by $f_{i,w}$, where $f_{i,w}\coloneqq (F_i|_{H_{w}},o_z)\in\mathcal{O}_{\mathbb{C}^n,o}$ for $i=1,\ldots,t$. Then we can see that as a linear space over $\mathbb{C}$, $\mathcal{L}_w$ is spanned by
    \[\mathcal{S}_w\coloneqq \Big\{\Pi(z^{\alpha}f_{i,w})\colon \alpha\in\mathbb{N}^n, \ |\alpha|\le N-1, \ i=1,\ldots,t\Big\}.\]
    
    We denote the coefficient matrix of $\mathcal{S}_w$ under the basis $\mathcal{B}$ by
    \[\mathbf{A}(w)=\Big(a_{\alpha j}(w)\Big)_{\substack{1\le |\alpha|\le N-1 \\ 1\le j\le q}},\]
    where $q\coloneqq t\cdot\binom{n}{N+n-1}$ denotes the number of elements in $\mathcal{S}_w$. Then Lemma \ref{lem-linear.algebra} gives the corresponding open subset $U$ of $\Delta^m$ and the matrix
    \[\mathbf{B}(w)=\Big(b_{k\alpha}(w)\Big)_{\substack{1\le k\le s\\ 1\le |\alpha|\le N-1}}\]
    on $\Delta^m$ such that the three conditions in Lemma \ref{lem-linear.algebra} hold. Set $\xi_{k}(\cdot)=\big(\xi_k(w)_{\alpha}\big)_{\alpha\in\mathbb{N}^n}$, $k=1,\ldots,s$, where
    \begin{equation*}\big(\xi_{k}(w)\big)_{\alpha}=\left\{
        \begin{matrix}
            b_{k\alpha}(w) & |\alpha|\le N-1\\
            0 & |\alpha|\ge N
        \end{matrix}\right..
    \end{equation*}
    Then the four statements can be checked to be true by the results of Lemma \ref{lem-linear.algebra}. The statements (i)(ii)(iii) are clear. 
    
    For the last statement, now let $w\in U$. By the construction, for every $f\in I_w$, since $\mathbf{B}(w)\cdot\mathbf{A}(w)=\mathbf{O}$ we can see $\big(\xi_i(w)\cdot\Pi(f)\big)(o)=0$, which implies $(\xi_i(w)\cdot f)(o)=0$, for each $i\in\{1,\ldots,s\}$. Conversely, if $f\in\mathcal{O}_{\mathbb{C}^n,o}$ satisfies $\big(\xi_i(w)\cdot f\big)(o)=0$ for $i=1,\ldots,s$, then the column vector of the coefficients of $\Pi(f)$ under the basis $\mathcal{B}$ is in $\mathrm{Ker}(\mathbf{B}(w))$. As $\mathrm{Ker}(\mathbf{B}(w))=\mathrm{Im}(\mathbf{A}(w))$, it follows that $\Pi(f)\in\mathrm{Im}(\mathbf{A}(w))$, yielding $\Pi(f)\in I_w/\mathfrak{m}^N$. Thus, $f\in I_w$ as $\mathfrak{m}^N\subseteq I_w$.

    The proof is completed.
    \end{proof}

    We give the following lemma by \cite[Lemma 2.1]{BGY22}, while this lemma is also a direct consequence of the coherence of multiplier ideal sheaf when $\deg(\xi)<+\infty$.

    \begin{Lemma}\label{lem-K.xi.equal.0}
        Let $\varphi$ be a negative holomorphic function on $\Delta^n$, and $\xi\in\ell_1^{(n)}$. Denote by $K_{\xi,\Delta^n}^{\varphi}(\cdot)$ the Bergman kernel on $\Delta^n$ with respect to the functional $\xi$ and the weight function $\varphi$. Then $K_{\xi,\Delta^n}^{\varphi}(o)=0$ if and only if $(\xi\cdot f)(o)=0$ for every $(f,o)\in\mathcal{I}(\varphi)_o$.
    \end{Lemma}

    \begin{proof}
        If $(\xi\cdot f)(o)=0$ for every $(f,o)\in\mathcal{I}(\varphi)_o$, since every $f\in A^2\big(\Delta^n, e^{-\varphi}\big)$ satisfies $(f,o)\in \mathcal{I}(\varphi)_{o}$, then we have $K_{\xi,\Delta^n}^{\varphi}(o)=0$ by the definition of $K_{\xi,\Delta^n}^{\varphi}(o)$.
        
        Conversely, we assume there exists some $(f_0,o)\in\mathcal{I}(\varphi)_o$ with $(\xi\cdot f_0)(o)\neq 0$. Then according to \cite[Lemma 2.1]{BGY22}, we can find some $F_0\in A^2(\Delta^n,e^{-\varphi})$ such that $(\xi\cdot F_0)(o)\neq 0$. Thus, we get $K_{\xi,\Delta^n}^{\varphi}(o)\ge |(\xi\cdot F_0)(o)|^2/\int_{\Delta^n}|F_0|^2e^{-\varphi}>0$.
    \end{proof}

We now prove Corollary \ref{cor:pluripolar}.

    \begin{proof}[Proof of Corollary \ref{cor:pluripolar}]
        Since the desired result is local, shrinking $\Delta^{n+m}_{(z,w)}$ if necessary, we may assume that $\phi$ is negative on $\Delta^{n+m}_{(z,w)}$, and $\mathscr{I}$ is globally finitely generated (just like the assumption before Lemma \ref{lem-construct.xi-i} for $\mathscr{I}$). Denote $I_w\coloneqq (\mathscr{I}|_{H_{w}})_{(o_z,w)}$ for every $w$. We firstly prove this corollary under the following additional assumption:
    
    \emph{there exists a positive integer $N$ such that $\mathfrak{m}_w^N\subseteq I_w$ for every $w\in\Delta^m_{w}$, where $\mathfrak{m}_{w}$ denotes the maximal ideal of the ring $\mathcal{O}_{H_w,(o_z,w)}$}.

    In this case, by Lemma \ref{lem-construct.xi-i}, there exists an open subset $U_N$ of $\Delta^m_w$ and a collection of functional-valued functions $\{\xi_1(\cdot),\ldots,\xi_s(\cdot)\}$ on $\Delta^m_w$ such that

    (i) $\Delta_w^m\setminus U_N$ is an analytic subset of $\Delta_w^m$;
    
    (ii) $\xi_i(w)\in\ell_1^{(n)}$ and $\deg(\xi_i(w))\le N-1$, for each $i\in\{1,\ldots,s\}$ and $w\in \Delta_w^m$;

    (iii) each $\xi_i(\cdot)$ is holomorphic on $\Delta_w^m$, $i\in\{1,\ldots,s\}$;

    (iv) for every $w'\in U_N$,
    \begin{equation*}
        I_{w'}=\big\{f\in\mathcal{O}_{\mathbb{C}^n,o}\colon (\xi_i(w')\cdot f)(o)=0, \ i=1,\ldots,s\big\}.
    \end{equation*}

    Now we set
    \[\Psi_N(w')\coloneqq \sup_{1\le i\le s} \log K_{\xi_i(w'),H_{w'}}^{\phi|_{H_{w'}}}(o_{z}), \ w'\in \Delta^m_w.\]
    According to Theorem \ref{log-psh-wrt-xi}, we have that $\log \left(K_{\xi_i,H_{w'}}^{\phi|_{H_{w'}}}(o_{z})\right)$ is plurisubharmonic in $w'\in\Delta^m_w$ for every $i\in\{1,\ldots,s\}$, and thus $\Psi_N(w')$ is also a plurisubharmonic function on $\Delta^m_w$. We verify
    \begin{equation}\label{eq-Lambda.eq.Psi.-1.-infty}
        \Lambda(\phi,\mathscr{I})\cap U_N=\Psi_N^{-1}(-\infty).
    \end{equation}
    Let $w'\in U_N$. If $w'\in \Lambda(\phi,\mathscr{I})$, then $\mathcal{I}(\phi|_{H_{w'}})_{(o_z,w')}\subseteq I_{w'}$, which implies $(\xi_i\cdot f)(o_z)=0$ for every $(f,o_z)\in \mathcal{I}(\phi|_{H_{w'}})_{(o_z,w')}$ and $i=1,\ldots, s$. Lemma \ref{lem-K.xi.equal.0} indicates $\log K_{\xi_i(w'),H_{w'}}^{\phi|_{H_{w'}}}(o_{z})=-\infty$ for each $i$. We get $\Psi_N(w')=-\infty$. This shows $\Lambda(\phi,\mathscr{I})\cap U_N\subseteq\Psi_N^{-1}(-\infty)$. And we can also get the opposite direction by Lemma \ref{lem-K.xi.equal.0}.

    Now (\ref{eq-Lambda.eq.Psi.-1.-infty}) shows 
    \[\Lambda(\phi,\mathscr{I})\subseteq \Psi_N^{-1}(-\infty)\cup\big(\Delta^m\setminus U_N\big).\]
    Since $\Delta^m\setminus U_N$ is an analytic subset of $\Delta^m_{w}$ and $\Psi_N$ is plurisubharmonic on $\Delta^m_w$, it follows that $\Lambda(\phi,\mathscr{I})$ is either the whole $\Delta^m_w$ or a pluripolar subset of $\Delta^m_w$.

    Finally, let $\mathscr{I}$ be a general locally generated ideal sheaf. for every $N\ge 1$, define
        \begin{flalign*}
            \begin{split}
            \Lambda_N(\phi,\mathscr{I})\coloneqq &\left\{w'\in\Delta^m_w\colon \mathcal{I}\left(\phi|_{H_{w'}}\right)_{(o_z,w')}\subseteq (\mathscr{I}|_{H_{w'}})_{(o_z,w')}+\mathfrak{m}_{w'}^N\right\}\\
            =& \Lambda\big(\phi,\mathscr{I}+\mathcal{I}^N_{\Delta_w^m}\big),
            \end{split}
        \end{flalign*}
    where $\mathcal{I}_{\Delta^m_w}$ denotes the ideal sheaf along the subvariety $\Delta^m_w$ on $\Delta^{m+n}_{(z,w)}$. By the previous discussions, every $\Lambda_N(\phi,\mathscr{I})$ is either the whole $\Delta^m_w$ or a pluripolar subset of $\Delta^m_w$. Additionally, according to Lemma \ref{lem-Krull.lem} (Krull's Lemma), we can see
    \begin{equation}\label{eq-using.Krull.lem}
        \Lambda(\phi,\mathscr{I})=\bigcap_{k=1}^{\infty}\Lambda_N(\phi,\mathscr{I}).
    \end{equation}
    It follows that $\Lambda(\phi,\mathscr{I})$ is also either the whole $\Delta^m_w$ or a pluripolar subset of $\Delta^m_w$.

    Especially, if $I_w=I$ is invariant, by the proofs of Lemma \ref{lem-linear.algebra} and Lemma \ref{lem-construct.xi-i}, the set $U_N$ can be chosen as the whole $\Delta^m_w$. Thus, according to (\ref{eq-Lambda.eq.Psi.-1.-infty}) and (\ref{eq-using.Krull.lem}), the set $\Lambda(\phi,\mathscr{I})$ for such $\mathscr{I}$ is either the whole $\Delta_w^m$ or a complete pluripolar subset of $\Delta_w^m$. More precisely, in this case $\Lambda_N(\phi,\mathscr{I})=\Psi_N^{-1}(-\infty)$ holds for each $N\in\mathbb{Z}_+$. And on every relatively compact subset of $\Delta^m_w$, one can select a positive real number $c_N$ such $c_N\Psi_N$ is upper bounded by $1$ for each $N$. Taking $\Psi\coloneqq \sup_{N}c_N\Psi_N$, we conclude that $\Lambda(\phi,\mathscr{I})=\Psi^{-1}(-\infty)$ locally. This verifies the desired result.
    \end{proof}
 
\vspace{.1in} {\em Acknowledgements}. We are grateful to Professor Bo Berndtsson for bringing this problem to our attention and giving us many helpful discussions. We would also like to thank Dr. Zhitong Mi for carefully checking this note. The second named author was supported by National Key R\&D Program of China 2021YFA1003100, NSFC-11825101 and NSFC-12425101. The third named author was supported by China Postdoctoral Science Foundation BX20230402 and 2023M743719.

\bibliographystyle{alpha}
\bibliography{xbib}

\end{document}